\documentclass[11pt]{amsart}
\usepackage{amsmath,amssymb,amscd}
\usepackage{breqn}

\usepackage{arydshln}

\usepackage{amsmath,amssymb}

 \usepackage{color}
\usepackage{amsxtra, amsmath}
\usepackage{amssymb, amscd}
\usepackage{graphicx}

\newcommand\NN{\mathbb N}

\theoremstyle{plain}
\newtheorem{thm}{Theorem}[section]
\newtheorem{lem}[thm]{Lemma}

\theoremstyle{definition}

\theoremstyle{remark}

\title[\sc Time and band limiting for exceptional polynomials]{Time and band limiting for exceptional polynomials}

\author{Castro M. M., Gr\"unbaum F. A. and  Zurri\'an I.}
\address{Department of Mathematics, University of California, Berkeley
CA 94705}
\email{grunbaum@math.berkeley.edu}
\address{Departamento de Matemática Aplicada II, Universidad de Sevilla, Seville 41011}
\email{mirta@us.es}
\address{Departamento de Matemática Aplicada II, Universidad de Sevilla, Seville 41011}
\email{ignacio.zurrian@fulbrightmail.org}

\date{\today}

\thanks{}



\begin{document}

\begin{abstract} 
	 The ``time-and-band limiting" commutative property was found and exploited	by  D. Slepian, H. Landau and H. Pollak at Bell Labs in the 1960's,  and independently by M. Mehta and later by C. Tracy and H. Widom in Random matrix theory. The property in question is the existence of local operators with simple spectrum that commute with naturally appearing global ones.

	  Here we give a general result that insures the existence of a commuting differential operator for a given family of exceptional orthogonal polynomials satisfying the ``bispectral property". As a main tool  we go beyond bispectrality and make use of the notion of  Fourier Algebras associated to the given sequence of exceptional polynomials.
	We illustrate this result with two examples, of Hermite and Laguerre type, exhibiting also a nice Perline's form for the commuting differential operator.
\end{abstract}
\maketitle

\section{Introduction}

The time-and-band limiting problem has its origins in signal processing, and was motivated by a question posed by Claude Shannon in  \cite{Sh}: what is the best use one can make of the values of the Fourier transform $\mathcal{F}f(k)$, for values of $k$ in the band 
$\left[- \mathcal W,\mathcal W\right]$, to recover   the time-limited signal $f(x)$ with finite support $\left[-\mathcal T,\mathcal T\right]$.
 The solution was given by three workers at Bell Labs during the 1960s: David Slepian, Henry Landau, and Henry Pollak. See \cite{SLP1,SLP2,SLP3,SLP4,SLP5,S1,S2}.

These papers show the central role of an integral operator $S$- with a kernel built from the exponentials that feature in the Fourier transform - which exhibits a ``spectral gap". Only about $2\mathcal W\times2\mathcal T$ of its eigenvalues are truly nonzero and one should find the projection of $f(x)$ on the span of the corresponding eigenfunctions. Any effort to compute the components of $f(x)$ along the remaining eigenfunctions will result in numerical instability. In practice the eigenfunctions of the integral operator in question can only be computed because they happen to be the eigenfunctions of a commuting differential operator $T$. This unexpected algebraic accident has played a role in other areas, even as far afield as the study of the Riemann zeta function, see \cite{CM}. For a further discussion of this and other historical points the reader can consult \cite{CG23}.

\bigskip
 
 The ``bispectral problem" introduced in \cite{DG} is an effort to understand the reason for the existence of this commuting operator $T$ and to extend this commutativity miracle.
The are numerous papers linking bispectrality with the building of a commuting operator, the most recent ones may be \cite{CGYZ1,CGYZ2,CGYZ3,CGYZ4,CG23}.

One has a ``bispectral situation" when a function $f(x,n)=q_n(x)$ is a simultaneous eigenfunction of an operator $D$, acting in the variable $x$, and an operator $L$, acting in the variable $n$. Namely,
\begin{equation*}
Dq_n(x) = \lambda_nq_n(x),\qquad
Lq_n(x)=\mu_xq_n(x),
\end{equation*}
where the eigenvalue $\lambda_n$ (respectively $\mu_x$) depends only on $n$ (respectively $x$). Both variables may be discrete or continuous. In this work we shall consider a continuous-discrete setup, given by a family of orthogonal polynomials that are eigenfunctions of a common differential operator. Previous works in the same direction include \cite{G3}, where the situation of the classical orthogonal polynomials is considered, and for matrix valued orthogonal polynomials we point out \cite{GPZ3,CG5, CG6, CGPZ, GPZ1,GPZ2}.
As a continuation of the work started in \cite{CG23}, in this paper we deal with sequences of exceptional orthogonal polynomials in connection with time-and-band limiting. We manage to get results that were not reached in \cite{CG23}.

\bigskip

 The relationship between exceptional polynomials (to be defined in Section \ref{Preliminares}), bispectrality and the use of the Darboux process has been considered by different authors, \cite{
GU2,%09
Qu,%09
STZ,%10
GUHerm,%14
GUKKM,%16
GGU2,%19
KM,%20
CGYZ4,
CG23}.  
The study of exceptional orthogonal polynomials has grown into a very active area in recent years. The literature is large and we just mention a few papers, \cite{Durec,Du,GGU2,GUHerm,GU2,GUKKM,KM,Qu,STZ}. A more detailed account  is given in \cite{CG23}.

 \section{The contents of the paper}\label{structure}

The case of exceptional Jacobi Polynomials was already covered in \cite[Section 5]{CG23}, where an effort is made to write down the commuting operator in ``Perline form" as explained now.

Quite a while back the case of polynomials defined on finite sets was considered in \cite{P1,P2,Per}. In \cite{Per} R. Perline obtained a simple form for the commuting local operator. 
 It was subsequently seen, see \cite{GPZ1, GVZ} that this simple form of Perline applies unchanged to
 other situations. 
In \cite{GVZ} one shows that a more complicated form
 of the commuting operator, still built with an appropriate extension of  Perline's construction, yields one with simple spectrum.
 
We will see in this paper that to obtain a Perline form of the commuting operator one has to go beyond bispectrality and make use of the more general  notion of Fourier Algebras (see Section \ref{FA} below), first considered  in \cite{CY1} in connection with bispectral meromorphic functions and the commutativity property and in \cite{CY2} in the framework of  matrix-valued orthogonal polynomials. This concept was also exploited in \cite{CGYZ1,CGYZ2,CGYZ3,CGYZ4}.
 
Specifically, in Theorem \ref{Thm} we give a general result that guarantees the existence of a commuting differential operator $T$ for an integral operator $S$ built from a family of exceptional orthogonal polynomials satisfying the bispectral conditions (\ref{bis}) below. The main tool turns out to be the computation of the dimensions of certain subspaces of the Fourier Algebras associated with the given sequence of exceptional polynomials (see Section \ref{FA}). The method of considering the Fourier algebras to find a commuting operator is not new and was used already in \cite{CGYZ4}. Here, we dive into a more detailed study and we proceed with the considerations needed in the particular case given by exceptional orthogonal polynomials, which do not fit neatly with the general results.

\section{Preliminaries}\label{Preliminares}

In this paper we deal with sequences of real valued polynomials \{$q_n(x)\}_{n\in\NN_0}$ for which there exist a nontrivial differential operator $D$ and a nontrivial difference operator $L$ such that:
\begin{equation}\label{bis}
Dq_n(x) = \lambda_nq_n(x),\qquad
Lq_n(x)=\mu_xq_n(x).
\end{equation}
We allow $q_n\equiv0$ only for $n$ in a finite set $X$, these are the so called {\it exceptional degrees}. On the other hand, we require $\deg({q_n})=n$ for $n\in Z:=\mathbb N_0\setminus X$ and \{$q_n(x)\}_{n\in Z}$ to be a complete orthonormal system with respect to a weight $w(x)$ defined on a (possibly infinite) open interval $(a,b)$. 
In the particular situation when $\operatorname{ord}(D)=2$, we are dealing with what are nowadays called ``exceptional orthogonal polynomials".
 
\bigskip

We consider the Hilbert spaces $\ell^2(Z)$ and  $L^2(w)=L^2((a,b), w(x)dx)$ given by the real valued sequences
$\{\gamma_n\}_{n\in Z}$ such that $\sum_{n\in Z} \gamma_n^2 < \infty$  and all measurable functions $f(x)$, $x\in (a,b)$, satisfying $\int_a^b |f(x)|^2w(x)\,dx < \infty $, respectively.
A natural   analog of the Fourier transform is  $F:\ell^2(Z) \longrightarrow L^2(w)$ given by
\begin{equation}\label{F}
\{\gamma_n\}_{n\in Z} \overset{F}{\longrightarrow} \sum_{n\in Z} \gamma_n q_n(x).\end{equation}
Its adjoint $F^*: L^2(w)\longrightarrow \ell^2(Z) $ is given by
$$ f \overset{F^*}{\longrightarrow} \gamma_n=\int_a^b f(x)\,w(x)\, q_n(x) dx.$$

In the spirit of Shannon we consider the problem of determining a function \( f \) from the following data: \( f \) is supported on the compact set \(\{n\in Z:\, n\leq N\}\), for a value $N\in Z$, and its transform \( Ff \) is known on a compact set \([a,\Omega]\), for $\Omega\in (a,b)$. This leads us to compute the singular vectors (and values) of the operator \( E: \ell^2(Z) \longrightarrow L^2(w) \) defined by
$$E f= \chi_\Omega F \chi_N f,$$
where \( \chi_N \) represents the time-limiting operator on \( \ell^2(Z) \) and \( \chi_\Omega \) represents the band-limiting operator on \( L^2(w) \).
Namely, \( \chi_N \) acts on \( \ell^2(Z) \) by zeroing out all components with indices greater than \( N \). On the other hand,  \( \chi_\Omega \) acts on \( L^2(w) \) by multiplying by the characteristic function of the interval \((a, \Omega)\).

We are thus required to study  the eigenvectors   of the operators
$$E^*E= \chi_N F^{*} \chi_\Omega F   \chi_N\qquad \text{ and } \qquad E E^*= \chi_\Omega F \chi_N F^{*} \chi_\Omega.$$

The operator $E^*E$, acting on $\ell^2(Z)$, is just a finite dimensional matrix with each  entry given by
\begin{equation}\label{MatrixM}
 (E^*E)_{m,n}= \int_a^\Omega q_m(x) w(x) q_n(x)  dx, \qquad  m,n \in Z,\quad m,n  \leq N.
\end{equation}

The operator $E E^*$ acts on $L^2((a,\Omega), w(x)dx)$ by means of the integral kernel
\begin{equation}\label{kernel}
  k(x,y)=\sum_{ n\in Z,\,n\leq N} q_n (x)q_n(y),
\end{equation}
namely,  the integral operator $S=EE^*$   is given by
\begin{equation}\label{intoper}
   (Sf)(x)=\int_{a}^\Omega f(y)w(y)k(x,y)dy.
\end{equation}

This is the analog of the integral operator considered by Slepian, Landau and Pollak.

\bigskip

For general \(N\) and \(\Omega\), finding the eigenfunctions of \(EE^*\) and \(E^*E\) analytically is an impossible task. However, a miracle comes to our rescue: we identify a local operator with simple spectrum that shares the same eigenfunctions as the operators \(E E^*\) or \(E^* E\). This approach was successfully employed by Slepian, Landau, and Pollak in the scalar case when dealing with traditional Fourier analysis. They made the following discoveries:

\begin{itemize}
  \item For each \(N\) and \(\Omega\), a symmetric tridiagonal matrix \(\hat T\) exists with a simple spectrum that commutes with \(E^*E\).
  \item For each \(N\) and \(\Omega\), there exists a self-adjoint second-order differential operator \(T\) with a simple spectrum that commutes with the integral operator \(S=EE^*\).
\end{itemize}

In the case of  traditional Fourier analysis, the operators $D$ and $L$ involved in \eqref{bis} may be taken of very low order. In this work, we need to consider arbitrary orders, which leads us to the construction of a higher order differential operator $T$.

\bigskip

For a full discussion  of important numerical aspects in the traditional Fourier case, see \cite{ORokX}.

\section{The Fourier Algebras}\label{FA}

A {\it difference operator} of order $2k\geq0$ is denoted by ${\tilde L}=\sum_{j=-k}^k {\tilde L}_j\,\delta^j$, with ${\tilde L}_k$ or ${\tilde L}_{-k}$ nonzero, where $\delta$ is the usual shift. The operator ${\tilde L}$ acts on functions $h:\mathbb N_0\longrightarrow \mathbb R$ by
\begin{equation}\label{difop}
({\tilde L}h)_n=\sum_{j=-k}^k {\tilde L}_{n,n+j}\, h_{n+j},
\end{equation}
with the understanding that ${\tilde L}_{n,j}=0$ for $j<0$.

\bigskip 
Let us consider the Fourier algebras associated with $q_n(x)$ defined by
$$
\mathcal  F_x=\{ \text {differential operators }  \tilde D :  \tilde Dq_n(x)=\tilde L q_n(x) \text{ for some difference operator } \tilde L, \text{ for all } n,\, x
\},
$$
$$
\mathcal  F_n=\{ \text {difference operators }  \tilde L : \tilde Dq_n(x)=\tilde L q_n(x) \text{ for some differential operator } \tilde D, \text{ for all } n,\, x
\}.
$$

By a slight abuse of notation, given any $\tilde L \in \mathcal  F_n$ we will also denote by $\tilde L$ the semi-infinite $(2k+1)$-diagonal matrix whose entries are given by  \eqref{difop}. We say that the operator $\tilde L $ is symmetric  if the corresponding matrix is symmetric.
For any $x\in \mathbb (a,b)$, let us consider the  semi-infinite vector $\mathbf{ q}(x):=\left(q_0(x),q_1(x),\dots\right)^T$. Then, if we have $\tilde Dq_n(x)=\tilde L q_n(x)$ for some $\tilde D\in\mathcal  F_x$,  in vector notation we will have
$$\tilde D \mathbf{ q}(x)= \tilde L \mathbf{ q}(x).
$$

It is worth noticing that we have made an important notational choice: the vector $\mathbf q(x)$ has zeros in every entry corresponding to an exceptional degree. The reader will realize that we could have skipped all these entries and that the results and statements will hold with the due technical adaptations. Nevertheless, we have made this choice resulting in a more succinct and shorter presentation, not only in the proofs in this section, but also in the examples in Sections \ref{EjemploHermite} and \ref{Ejemplo Laguerre}.
The matrix representation of $\tilde L$ in the basis obtained by including only the $q_n$ with $n$  in $Z$ is obtained by ``compressing" the matrices going with the choice above, i.e. we need to throw away the rows and columns corresponding to the exceptional degrees.

\bigskip

From \eqref{bis} we know that $D\in \mathcal  F_x$ and, since $\{q_n(x)\}_{n\in Z}$ is a orthogonal complete system, that $D$ is symmetric in $L^2(w)$.
This leads us to consider, for given $d$ and $k$ in $\mathbb N_0$, the subspaces 

\[
\mathcal  F_{x,sym}^{d,k}=\left\{
\text{Differential operators } \tilde D:
 \begin{aligned}
\tilde Dq_n(x)=\tilde L q_n(x) \text{ for a symmetric difference operator } \tilde L,
\\
\text{ with } \operatorname{ord}(\tilde D)\leq 2d   \text{ with }     \operatorname{ord}(\tilde L)\leq 2k.
\end{aligned}
\right\}
\]
Notice that if $\tilde Dq_n(x)=\tilde L q_n(x) $ with $\tilde L$ symmetric, since $\{q_n(x)\}_{n\in Z}$ is a orthogonal complete system, $\tilde D$ is a symmetric differential operator in $L^2(w)$. Hence, all the operators in $\mathcal  F_{x,sym}^{d,k}$ are symmetric in $L^2(w)$.
Recall that an operator $\tilde D$ is symmetric in $L^2(w)=L^2((a,b),w(x)dx)$  if for every $f$ and $g$ in an appropriate dense domain we have
$$\int_a^b \tilde D f(x) w(x) g(x) dx=\int_a^b  f(x) w(x) \tilde D g(x) dx.$$
One defines in an analogous manner the notion of symmetry for a differential operator $\tilde D$ in $L^2((a,\Omega),w(x)dx)$.

\bigskip

The identity above requires the vanishing of several terms resulting from integration by parts. This can come about either with the help
of certain coefficients in the operator $\tilde D$ or because the weight $w(x)$ and some of its derivatives vanish at the end points. In the absence of this, one needs to restrict the domain of the operator to functions $f(x),g(x)$ which, maybe along with their derivatives, vanish at these end points. One should keep this remark in mind when reading the comment after the proof of the lemma below.

\bigskip

 Now we give the general results that allow us to find commuting operators in our time-and-band limiting setup.

\begin{lem}\label{Lem} If $T\in\mathcal  F_{x,sym}^{d,k}$ is also symmetric in $L^2((a,\Omega),w(x)dx)$ and  $\deg(Tq_n(x))\leq N$ for any $n\leq N$, then $T$ commutes with the integral operator $S$ in \eqref{intoper}.
\end{lem}
\begin{proof}  If we denote by $I_N$ the semi-infinite diagonal matrix whose entries $(I_N)_{j,j}$ are $1$ for $0\leq j\leq N$ and zero elsewhere, we have
 $$  k(x,y)=\sum_{ n\in Z,\,n\leq N} q_n (x)q_n(y)=\mathbf{ q}(x)^T\,I_N \, \mathbf{ q}(y).
$$
By hypothesis, there is a symmetric semi-infinite $(2k+1)$-diagonal matrix $\hat T$ such that $T\mathbf{ q} (x)= \hat T\mathbf{ q} (x)$.  This, combined with 
the fact $\deg(Tq_n(x))\leq N$ for any $n\leq N$, gives us $ \hat T\,I_N=I_N\, \hat T$.
 Hence,
\begin{equation}\label{Tx}
T(x)(k(x,y))= ({ \hat T} \mathbf{ q}(x))^T\,I_N \, \mathbf{ q} (y)=\mathbf{ q}(x)^T\,I_N \,{ \hat T} \mathbf{ q}(y)=T(y)(k(x,y)).\end{equation}
Notice that to avoid ambiguity in the previous identities, we have used $T(x)$ and $T(y)$ to denote the action of the differential operator in the variable $x$ and $y$, respectively.

Now, let $f:(a,b)\longrightarrow \mathbb R$ be a sufficiently differentiable function in the domain of $T$, then 
\begin{align*}
T(S(f)(x))
&=  \int_{a}^\Omega f(y)w(y)T(x)k(x,y)dy\stackrel{i}{=} \int_{a}^\Omega f(y)w(y)T(y)k(x,y) dy\\
&\stackrel{ii}{=}  \int_{a}^\Omega Tf(y)w(y)k(x,y) dy{=}(S(Tf)(x)),
\end{align*}
with $(i)$ given by \eqref{Tx} and $(ii)$  a consequence of $T$ being symmetric with respect to $w$ in $(a,\Omega)$. 
This completes the proof.\end{proof}

\bigskip

One more comment before our main result. Since the integral operator  $S$ has as its domain all of 
$L^2((a,\Omega),w(x)dx)$ it is important that the commuting differential operator that we produce should have a large domain
that includes the eigenfunctions of $S$. These eigenfunctions, or their derivatives, do not vanish at either end point. The weight
$w(x)$ will not vanish at $\Omega$ and thus the symmetry of $T$ will require conditions such as \eqref{cond1} below.

\bigskip

\begin{thm}\label{Thm}
If for some  $d,k$ in $\mathbb N_0$, we have that 
\begin{equation}\label{cond}
\dim\left(\mathcal  F_{x,sym}^{d,k}\right)>\frac{k(k+1)+d(d+1)}{2}+1,\end{equation}
then there exists a nonscalar operator $T\in\mathcal  F_{x,sym}^{d,k}$  that commutes with the integral operator $S$ in \eqref{intoper}.
\end{thm}
\begin{proof}
Given an arbitrary $T\in\mathcal  F_{x,sym}^{d,k}$ (with real coefficients), we can express it as $$T=w(x)^{-1/2}\left(\sum_{j=0}^d\,\partial_x^j\,t_j(x)\,\partial_x^j\right) w(x)^{1/2},$$
for some functions $t_j$. Given that it is symmetric with respect to $w$ in $(a,b)$, if we also want it to be symmetric in $(a,\Omega)$ it is enough to impose the following $\frac{d(d+1)}{2}$ linear conditions: 

\begin{equation}\label{cond1}
t_j^{(i)}(\Omega)=0,\quad \text{ for } j=1,\dots,d,\quad  i=0,\dots,j-1.
\end{equation}

Since $T\in\mathcal  F_{x,sym}^{d,k}$, we  have a symmetric semi-infinite $(2k+1)$-diagonal matrix $\hat T$ such that $T\mathbf{ q} (x)= \hat T\mathbf{ q} (x)$, 
\begin{equation}\label{t-hat}
\hat T=
\left[
    \begin{array}{ccccc ;{2pt/2pt}ccc}
 {\hat T}_{00}& {\hat T}_{01}&\dots  &   {\hat T}_{0k}& &            &    &      \\
{\hat T}_{10}   &\ddots&\ddots &\ddots     &\ddots&         &    &      \\
\vdots   &\ddots&\ddots &\ddots     &\ddots&         {\hat T}_{N+1-k,N+1}     &    &      \\
{\hat T}_{k0} &\ddots&\ddots      &\ddots   &\ddots&\vdots     &  \ddots  &      \\
          &\ddots&\ddots&\ddots&\ddots&         {\hat T}_{N,N+1}     &\dots &   {\hat T}_{N,N+k}   \\ \hdashline[2pt/2pt]
          & &\ddots&   \ddots&\ddots&          \ddots&          \ddots&          \ddots
    \end{array}
\right] .
\end{equation}
Hence, if we want $\deg(Tq_n(x))\leq N$ for any $n\leq N$, it is enough to satisfy the following $\frac{k(k+1)}{2}$ linear conditions:
\begin{equation}\label{cond2}
\hat T_{i,j}=0 ,\quad \text{ for } j=N+1,\dots,N+k,\quad  i=j-k,\dots,N,
\end{equation} 
i.e., the elements of the top right-hand side of the matrix \eqref{t-hat} should vanish.
We understand that $\hat T_{i,j}=0$ if $i$ is  negative.
 
Then, for $\mathcal  F_{x,sym}^{d,k}$ to contain a nonzero operator $T$ satisfying \eqref{cond1} and \eqref{cond2}, it is enough that its dimension should be at least $\frac{k(k+1)+d(d+1)}{2}+1$.  
If we require that $T$ should not be a scalar multiplication, it is sufficient that the dimension should be larger than $\frac{k(k+1)+d(d+1)}{2}+1$. By Lemma \ref{Lem}, any operator $T$ satisfying \eqref{cond1} and \eqref{cond2} commutes with $S$. This completes the proof
which is in part an extension of the one in \cite{Per}.
\end{proof}

In the following sections, we consider two different families of exceptional orthogonal polynomials.  We apply Theorem \ref{Thm} to show the existence of a commuting differential operator, and we use the concepts in Lemma \ref{Lem} to build it explicitly.

\bigskip

A careful look at the appropriate  selfadjoint extension of the symmetric commuting differential operator $T$ has, to the best of our knowledge, only been addressed in \cite{KV} and only for the case of traditional Fourier analysis.

\section{Hermite Exceptional Polynomials} \label{EjemploHermite}

Let us consider the sequence of Hermite exceptional polynomials given by 
$$q_n(x)=\frac{\sqrt{(n - 1) (n - 2)}}{\sqrt[4]{\pi}\sqrt{ 2^nn!}}
\left(H_n(x)+4nH_{n-2}(x)+4n(n-3)H_{n-4}(x)\right),\quad n\in\mathbb N_0, $$
where $H_n$ are the classical Hermite polynomials given by the Rodrigues formula

$$H_n=(-1)^ne^{x^2}D_x^ne^{-x^2}.$$

We have that $q_n\equiv0$ only for $n$ in $X=\{1,2\}$ and  \{$q_n(x)\}_{n\notin X}$ is a complete orthonormal system with respect to the weight 
$$w(x)=\frac{e^{-x^2}}{(1+2x^2)^2}, \quad x\in(-\infty,\infty).$$

Now we use Theorem \ref{Thm} to find a differential operator $T$ that commutes with the operator $S$ given by 
$$(Sf)(y)=\int_{-\infty}^\Omega \left(\sum_{k=0}^N q_k(x)q_k(y)\,w(x)\right)\, f(x)\, dx.$$
Namely, we look for $d$ and $k$ satisfying \eqref{cond}. 
Solving the system of equations $\tilde D q_n(x)=\tilde L q_n(x)$, with $\deg(\tilde D)\leq 2$ and $\deg(\tilde T)\leq 6$, one finds that $\mathcal  F_{x,sym}^{3,3}$ contains  the differential operators
\begin{align*}
D_0&=\dfrac{d^2}{dx^2}-\frac{2 \, {\left(2 \, x^{3} + 5 \, x\right)}}{2 \, x^{2}  +1}\dfrac{d}{dx},\\
D_1&=x\dfrac{d^2}{dx^2} -\frac{4 \, x^{4} + 8 \, x^{2} - 1}{2 \, x^{2} + 1}\dfrac{d}{dx},\\
D_2&=\left(-\frac{1}{56} \, x^{2} + \frac{17}{112}\right)\dfrac{d^2}{dx^2}+
\left(  \frac{4 \, x^{5} - 28 \, x^{3} - 87 \, x}{56 \, {\left(2 \, x^{2} + 1\right)}}\right)\dfrac{d}{dx}+1,\\
D_3&= \frac{4x^3}{3} + 2 x ,
\end{align*} with $D_jq_n(x)=L_jq_n(x)$ for 
\begin{align*}
L_0=&-2n,\\
L_1=&
	\sqrt{2} \sqrt{n - 1} \sqrt{n - 3} \sqrt{n} \delta^{-1}
	+\sqrt{2} \sqrt{n + 1} \sqrt{n - 2} \sqrt{n} \delta
,\\
L_2=&
	\frac{   \sqrt{{\left(n - 2\right)} {\left(n - 3\right)} {\left(n - 4\right)} n} }{56}\delta^{-2} 
	+\frac{   {{\left(n - 4\right)} {\left(n - 7\right)}} }{28}
	+ \frac{  \sqrt{{\left(n + 2\right)} {\left(n - 1\right)} {\left(n - 2\right)} n}}{56}\delta^{2}
,\\
L_3=& \frac{\sqrt{2}}{3}\sqrt{n(n-4)(n-5)} \delta^{-3}+\sqrt{2n(n-1)(n-3)}\, \delta^{-1}+\sqrt{2(n+1)n(n-2)}\, \delta\\
&+\frac{\sqrt{2}}{3}\sqrt{(n+3)(n-1)(n-2)} \delta^3.
\end{align*}

\bigskip

Furthermore, one can verify that the following $14$ operators are linearly independent:\begin{align*}Y&=\{
Id,D_0,D_1,D_2,D_3,D_0^2,D_0^3,D_1^2,D_3D_0+D_0D_3,D_3D_0^2+D_0^2D_3,D_0D_1+D_1D_0,D_3D_0^3+D_0^3D_3,\\ &D_0^2D_1+D_1D_0^2
,D_0D_1^2+D_1^2D_0\}.\end{align*} 
Since $Y\subset \mathcal  F_{x,sym}^{3,3}$ , and  $14>\frac{3(3+1)+3(3+1)}{2}+1$, by Theorem \ref{Thm} we know that there exists an operator $T\in\mathcal  F_{x,sym}^{3,3}$ that commutes with $S$.
To find such a $T$, we consider a linear combination of the operators in $Y$ that satisfies both hypothesis of Lemma \ref{Lem}. After solving the corresponding system of linear equations one obtains that\begin{align*}
T=	&- \, \Omega^{3}D_0^3+\tfrac{3}{4}\left( 2 \, \Omega^{2} - 1 \right)\left({D_1}D_0^2+D_0^2{D_1}\right)-\tfrac{3}{2} \, \Omega({D_1}^2D_0+D_0{D_1}^2)+\tfrac{3}{8} \left( {D_3}D_0^3+D_0^3{D_3}\right)\\
	&+3( \, N \Omega^{2} + 2\, \Omega^{2} - \tfrac{3}{2}  \, N - 4)({D_1}D_0+D_0{D_1})-3\, {\left(2 \, \Omega^{2} - 1\right)} \Omega \ D_0^2\\
	&-2\Omega \, {\left(54 \, N^{2} + 4 \, \Omega^{2} + 48 \, N - 3\right)}  D_0 \nonumber -6 \, {\left(2 \, N + 1\right)} \Omega\ {D_1}^2+\tfrac{9}{8} (2 \, N + 1)({D_3}D_0^2+D_0^2{D_3})\\
	&+\tfrac{3}{2} (3 \, N^{2} + 3 \, N - 7)(D_0{D_3}+{D_3}D_0) +3(6 \, N \Omega^{2} - 4 \, N^{2} + 2 \, \Omega^{2} - 17 \, N - 7){D_1}\\
	&+672{\left(N + 1\right)} N \Omega\ {D_2}+\tfrac{3}{2} {\left(2 \, N^{2} + 2 \, N - 15\right)} {\left(2 \, N + 1\right)}{D_3}\nonumber
\end{align*} satisfies the desired conditions. Then, $T$ commutes with $S$.

 \bigskip

This extension of the Perline representation of a commuting $T$ is not unique, and a different choice of enough linearly independent operators
in $\mathcal  F_{x,sym}^{3,3}$ would produce a different representation.

\bigskip
Returning to the operator $T$ given above, if one in interested in an explicit expression for the coefficients one can use 
$$T=w^{-1/2}\left(\sum_{k=0}^3 \partial_x^k f_k \partial_x^k\right)w^{1/2},$$ 
with 
\begin{dmath*}
f_3=(x-\Omega)^3
\end{dmath*}
\begin{dmath*}
f_2=
{3}{{\left(2 \, x^{2} + 1\right)}^{-2}}  {\left(\Omega - x\right)}^{2}\, \left(4 \, \Omega x^{6} - 4 \, x^{7} + 8 \, N x^{5} + 8 \, \Omega x^{4} - 12 \, x^{5} + 8 \, N x^{3} + 21 \, \Omega x^{2} - 17 \, x^{3} + 2 \, N x - 7 \, \Omega + 10 \, x\right),
\end{dmath*}
\begin{dmath*}
f_1=
-{3}{{\left(2 \, x^{2} + 1\right)}^{-4}}{\left(\Omega - x\right)} \, \left(16 \, \Omega^{2} x^{12} - 32 \, \Omega x^{13} + 16 \, x^{14} + 64 \, N \Omega x^{11} - 64 \, N x^{12} + 64 \, N^{2} x^{10} + 64 \, \Omega^{2} x^{10} - 160 \, \Omega x^{11} + 96 \, x^{12} + 192 \, N \Omega x^{9} - 256 \, N x^{10} + 160 \, N^{2} x^{8} + 168 \, \Omega^{2} x^{8} - 272 \, \Omega x^{9} + 72 \, x^{10} + 416 \, N \Omega x^{7} - 448 \, N x^{8} + 160 \, N^{2} x^{6} + 152 \, \Omega^{2} x^{6} - 192 \, \Omega x^{7} - 120 \, x^{8} + 160 \, N \Omega x^{5} - 128 \, N x^{6} + 80 \, N^{2} x^{4} - 279 \, \Omega^{2} x^{4} + 334 \, \Omega x^{5} - 359 \, x^{6} - 76 \, N \Omega x^{3} + 92 \, N x^{4} + 20 \, N^{2} x^{2} + 1210 \, \Omega^{2} x^{2} - 1662 \, \Omega x^{3} + 324 \, x^{4} - 36 \, N \Omega x + 32 \, N x^{2} + 2 \, N^{2} - 83 \, \Omega^{2} + 610 \, \Omega x - 497 \, x^{2} - 2 \, N + 18\right) ,
\end{dmath*}
\begin{dmath*}
f_0= {{\left(2 \, x^{2} + 1\right)}^{-6}}
	(64 \, \Omega^{3} x^{18} - 192 \, \Omega^{2} x^{19} + 192 \, \Omega x^{20} - 64 \, x^{21} + 384 \, N \Omega^{2} x^{17} - 768 \, N \Omega x^{18} + 384 \, N x^{19} + 768 \, N^{2} \Omega x^{16} + 384 \, \Omega^{3} x^{16} - 768 \, N^{2} x^{17} - 1344 \, \Omega^{2} x^{17} + 1536 \, \Omega x^{18} - 576 \, x^{19} + 512 \, N^{3} x^{15} + 1920 \, N \Omega^{2} x^{15} - 4608 \, N \Omega x^{16} + 2688 \, N x^{17} + 3456 \, N^{2} \Omega x^{14} + 624 \, \Omega^{3} x^{14} - 4224 \, N^{2} x^{15} - 1296 \, \Omega^{2} x^{15} + 336 \, \Omega x^{16} + 336 \, x^{17} + 2304 \, N^{3} x^{13} + 4128 \, N \Omega^{2} x^{13} - 8640 \, N \Omega x^{14} + 4000 \, N x^{15} + 50496 \, N^{2} \Omega x^{12} + 112 \, \Omega^{3} x^{12} - 10176 \, N^{2} x^{13} + 3552 \, \Omega^{2} x^{13} - 10896 \, \Omega x^{14} + 7616 \, x^{15} + 4224 \, N^{3} x^{11} + 3840 \, N \Omega^{2} x^{11} + 37824 \, N \Omega x^{12} - 1344 \, N x^{13} + 137184 \, N^{2} \Omega x^{10} - 1332 \, \Omega^{3} x^{10} - 12192 \, N^{2} x^{11} + 12156 \, \Omega^{2} x^{11} - 29052 \, \Omega x^{12} + 18996 \, x^{13} + 4160 \, N^{3} x^{9} - 1464 \, N \Omega^{2} x^{9} + 134928 \, N \Omega x^{10} - 10584 \, N x^{11} + 165360 \, N^{2} \Omega x^{8} + 1200 \, \Omega^{3} x^{8} - 7440 \, N^{2} x^{9} + 9804 \, \Omega^{2} x^{9} - 31272 \, \Omega x^{10} + 20940 \, x^{11} + 2400 \, N^{3} x^{7} + 11160 \, N \Omega^{2} x^{7} + 153312 \, N \Omega x^{8} - 9176 \, N x^{9} + 107784 \, N^{2} \Omega x^{6} + 17521 \, \Omega^{3} x^{6} - 1944 \, N^{2} x^{7} - 24975 \, \Omega^{2} x^{7} - 1917 \, \Omega x^{8} + 9851 \, x^{9} + 816 \, N^{3} x^{5} + 8670 \, N \Omega^{2} x^{5} + 109500 \, N \Omega x^{6} - 12858 \, N x^{7} + 39708 \, N^{2} \Omega x^{4} - 139845 \, \Omega^{3} x^{4} + 108 \, N^{2} x^{5} + 266580 \, \Omega^{2} x^{5} - 160785 \, \Omega x^{6} + 31338 \, x^{7} + 152 \, N^{3} x^{3} - 1044 \, N \Omega^{2} x^{3} + 51324 \, N \Omega x^{4} - 9792 \, N x^{5} + 7818 \, N^{2} \Omega x^{2} + 69189 \, \Omega^{3} x^{2} + 162 \, N^{2} x^{3} - 247383 \, \Omega^{2} x^{3} + 242145 \, \Omega x^{4} - 68367 \, x^{5} + 12 \, N^{3} x - 1134 \, N \Omega^{2} x + 11466 \, N \Omega x^{2} - 1964 \, N x^{3} + 642 \, N^{2} \Omega - 2097 \, \Omega^{3} + 24 \, N^{2} x + 27432 \, \Omega^{2} x - 56673 \, \Omega x^{2} + 29088 \, x^{3} + 582 \, N \Omega + 144 \, N x + 1134 \, \Omega - 1512 \, x).
\end{dmath*}

\bigskip

The construction above follows the general arguments given in the previous section. This is not the
only way to produce such a commuting operator $T$. One can avail oneself of enough many linearly independent difference operators in the Fourier algebra and, given $N$ and $\Omega$, look for a linear combination of these
	that would commute with the matrix of size $N$ of inner products of the exceptional polynomials in the interval $(-\infty,\Omega)$. When we replace in this linear combination the difference
	operators by their corresponding differential ones in the Fourier algebra we obtain a commuting $T$.

\section{Laguerre Exceptional Polynomials}\label{Ejemplo Laguerre}

Let $\alpha>0$ and $L^{(\alpha)}_n(x)$, $n\geq 0$, denote the classical Laguerre polynomials orthogonal with respect
to the weight $e^{-x}x^\alpha$ in the interval $(0,+\infty)$. We consider the sequence of exceptional Laguerre polynomials
$$q_n(x)=\sqrt{\frac{(\alpha+n-1)(n-1)!}{(\alpha+n)\Gamma(n+\alpha)}}\left(-(x+\alpha+1)L_{n-1}^{(\alpha)}+L_{n-2}^{(\alpha)}\right),\quad n\in\mathbb N_0.$$

We have that $q_n\equiv0$ only for $n=0$ and  \{$q_n(x)\}_{n>0}$ is a complete orthonormal system with respect to the weight 
$$w(x)=\frac{e^{-x}x^\alpha}{(x+\alpha)^2}.$$ 

In this section, just as we did in the Hermite case, we find a differential operator $T$ that commutes with the operator $S$ given by 
$$(Sf)(y)=\int_{0}^\Omega \left(\sum_{k=0}^N q_k(x)q_k(y)\,w(x)\right)\, f(x)\, dx.$$

In this situation, and proceeding as in the previous example, we find that $\mathcal  F_{x,sym}^{2,2}$ contains the   differential operators
\begin{align*}
D_0=& \frac{1}{4} \, x\dfrac{d^2}{dx^2}+ \frac{{\left(\alpha + x + 1\right)} {\left(\alpha - x\right)}}{4 \, {\left(\alpha + x\right)}}\dfrac{d}{dx}+  \frac{\alpha + 3 \, x}{4 \, {\left(\alpha + x\right)}},\\
D_1=&{\left(3 \, \alpha - x + 10\right)} \frac{x}{42}\dfrac{d^2}{dx^2}\\
	&+\frac{3 \, \alpha^{3} - \alpha^{2} x - 3 \, \alpha x^{2} + x^{3} + 13 \, \alpha^{2} - 5 \, \alpha x - 10 \, x^{2} + 10 \, \alpha - 10 \, x}{42 \, {\left(\alpha + x\right)}} \dfrac{d}{dx}\\
	&-\frac{\alpha^{3} - \alpha x^{2} + 5 \, \alpha^{2} - 3 \, \alpha x + 10 \, \alpha - 10 \, x}{42 \, {\left(\alpha + x\right)}},\\
D_2=&(x+\alpha)^2,
\end{align*}

 with $D_jq_n(x)=L_jq_n(x)$ for 
\begin{align*}
L_0=&(3-n)/4,\\ 
L_1=&
	-\frac{\sqrt{\alpha + n} \sqrt{\alpha + n - 1} \sqrt{\alpha + n - 2} \sqrt{n - 1}}{42}\delta^{-1} 
	+\frac{ {\left(n - 1\right)} {\left(n - 5\right)}}{21}\\
	 &-\frac{ \sqrt{\alpha + n + 1} \sqrt{\alpha + n} \sqrt{\alpha + n - 1} \sqrt{n}}{42}\delta ,\\ 
L_2=&
	\sqrt{\alpha + n} \sqrt{\alpha + n - 3} \sqrt{n - 1} \sqrt{n - 2} \,\delta^{-2}
	 -4 \, \sqrt{\alpha + n} \sqrt{\alpha + n - 1} \sqrt{\alpha + n - 2} \sqrt{n - 1} \,\delta^{-1}\\
	 &	+ (4  \alpha^{2} + 10 \, \alpha n + 6 \, n^{2} - 5 \, \alpha - 6 \, n)
	 -4  \sqrt{\alpha + n + 1} \sqrt{\alpha + n} \sqrt{\alpha + n - 1} \sqrt{n}\, \delta^{1}\\
	&+\sqrt{\alpha + n + 2} \sqrt{\alpha + n - 1} \sqrt{n + 1} \sqrt{n}\, \delta^{2}
, 
\end{align*}

In this case, we see that $\mathcal  F_{x,sym}^{2,2}$ already contains at least $8$  operators :
$$Y=\{
Id,D_0,D_1,D_2
,D_0^2,D_1^2,D_0D_1+D_1D_0,D_2D_0+D_0D_2\}
. $$
One can verify that the operators in $Y$ are linearly independent. Hence, since $8>\frac{2(2+1)+2(2+1)}{2}+1$, by Theorem \ref{Thm} we know that there exists an operator $T\in\mathcal  F_{x,sym}^{2,2}$ that commutes with $S$. 
As before, to find such a $T$, we consider a linear combination of the operators in $Y$ that satisfies the hypothesis of Lemma \ref{Lem}, obtaining
\begin{align*}
T=
	&1764 D_1^2
	+168(D_0D_1+D_1D_0)
	+16  \left(\Omega - 3 \, \alpha - 10\right)^{2}D_0^2
	+4\left(N + \alpha\right)\left(D_0D_2+D_2D_0\right)\\
	&+42 \, {\left(2 \, N \Omega + 2 \, N \alpha + 4 \, \alpha^{2} - 5 \, \Omega + 15 \, \alpha + 40\right)}D_1\\
	&+4\Big(N \Omega^{2} - 6 \, N \Omega \alpha - \Omega^{2} \alpha - 7 \, N \alpha^{2} + 2 \, \Omega \alpha^{2} - 13 \, \alpha^{3} - 20 \, N \Omega - 5 \, \Omega^{2} - 20 \, N \alpha + 30 \, \Omega \alpha - 85 \, \alpha^{2} \\
		&\quad\quad+ 90 \, \Omega - 270   \alpha - 400\Big) D_0
	+{\left(N + \alpha\right)} {\left(N - \alpha - 5\right)}D_2.
\end{align*}

An  explicit expression is given by
$$T=w^{-1/2}\left(\sum_{k=0}^2 \partial_x^k f_k \partial_x^k\right)w^{1/2},$$
with

\begin{dmath*}
f_2=\left(\Omega - x\right)^{2} x^{2},
\end{dmath*}
\begin{dmath*}
f_1=\left(\Omega - x\right){\left(\alpha + x\right)}^{-2}\Big(\Omega \alpha^{4} + 2 \, N \Omega \alpha^{2} x - 2 \, \Omega \alpha^{3} x + \alpha^{4} x + 4 \, N \Omega \alpha x^{2} - 2 \, N \alpha^{2} x^{2} - 6 \, \Omega \alpha^{2} x^{2} + 4 \, \alpha^{3} x^{2} + 2 \, N \Omega x^{3} - 4 \, N \alpha x^{3} - 2 \, \Omega \alpha x^{3} + 4 \, \alpha^{2} x^{3} - 2 \, N x^{4} + \Omega x^{4} - x^{5} + 3 \, \Omega \alpha^{3} - 11 \, \Omega \alpha^{2} x + 6 \, \alpha^{3} x - 19 \, \Omega \alpha x^{2} + 30 \, \alpha^{2} x^{2} - 5 \, \Omega x^{3} + 30 \, \alpha x^{3} + 6 \, x^{4} - 10 \, \Omega \alpha x + 38 \, \alpha^{2} x - 2 \, \Omega x^{2} + 96 \, \alpha x^{2} + 50 \, x^{3}\Big) ,
\end{dmath*}
 \begin{dmath*}
f_0=\left(\alpha + x\right)^{-3}
\Big(2 \, N^{2} \alpha^{5} - 4 \, N \Omega \alpha^{5} - 4 \, N \alpha^{6} - 9 \, \alpha^{7} + 10 \, N^{2} \alpha^{4} x - 8 \, N \Omega \alpha^{4} x - 8 \, N \alpha^{5} x - 25 \, \alpha^{6} x + 20 \, N^{2} \alpha^{3} x^{2} - 22 \, \alpha^{5} x^{2} + 20 \, N^{2} \alpha^{2} x^{3} + 8 \, N \Omega \alpha^{2} x^{3} + 8 \, N \alpha^{3} x^{3} - 6 \, \alpha^{4} x^{3} + 10 \, N^{2} \alpha x^{4} + 4 \, N \Omega \alpha x^{4} + 4 \, N \alpha^{2} x^{4} - \alpha^{3} x^{4} + 2 \, N^{2} x^{5} - \alpha^{2} x^{5} + 2 \, N \Omega^{2} \alpha^{3} - 32 \, N \Omega \alpha^{4} - 2 \, \Omega^{2} \alpha^{4} - 40 \, N \alpha^{5} + 13 \, \Omega \alpha^{5} - 86 \, \alpha^{6} + 10 \, N \Omega^{2} \alpha^{2} x - 88 \, N \Omega \alpha^{3} x - 10 \, \Omega^{2} \alpha^{3} x - 128 \, N \alpha^{4} x + 34 \, \Omega \alpha^{4} x - 249 \, \alpha^{5} x + 14 \, N \Omega^{2} \alpha x^{2} - 80 \, N \Omega \alpha^{2} x^{2} - 14 \, \Omega^{2} \alpha^{2} x^{2} - 154 \, N \alpha^{3} x^{2} + 24 \, \Omega \alpha^{3} x^{2} - 240 \, \alpha^{4} x^{2} + 6 \, N \Omega^{2} x^{3} - 24 \, N \Omega \alpha x^{3} - 6 \, \Omega^{2} \alpha x^{3} - 90 \, N \alpha^{2} x^{3} - 2 \, \Omega \alpha^{2} x^{3} - 82 \, \alpha^{3} x^{3} - 30 \, N \alpha x^{4} - 5 \, \Omega \alpha x^{4} - 10 \, \alpha^{2} x^{4} - 6 \, N x^{5} - 5 \, \alpha x^{5} - 80 \, N \Omega \alpha^{3} - 7 \, \Omega^{2} \alpha^{3} - 78 \, N \alpha^{4} + 97 \, \Omega \alpha^{4} - 374 \, \alpha^{5} - 240 \, N \Omega \alpha^{2} x - 43 \, \Omega^{2} \alpha^{2} x - 234 \, N \alpha^{3} x + 293 \, \Omega \alpha^{3} x - 1058 \, \alpha^{4} x - 240 \, N \Omega \alpha x^{2} - 57 \, \Omega^{2} \alpha x^{2} - 234 \, N \alpha^{2} x^{2} + 311 \, \Omega \alpha^{2} x^{2} - 946 \, \alpha^{3} x^{2} - 80 \, N \Omega x^{3} - 21 \, \Omega^{2} x^{3} - 78 \, N \alpha x^{3} + 115 \, \Omega \alpha x^{3} - 222 \, \alpha^{2} x^{3} + 40 \, \alpha x^{4} + 2 \, \Omega^{2} \alpha^{2} + 242 \, \Omega \alpha^{3} - 920 \, \alpha^{4} - 6 \, \Omega^{2} \alpha x + 790 \, \Omega \alpha^{2} x - 2684 \, \alpha^{3} x + 884 \, \Omega \alpha x^{2} - 2596 \, \alpha^{2} x^{2} + 320 \, \Omega x^{3} - 840 \, \alpha x^{3} - 1200 \, \alpha^{3} - 3600 \, \alpha^{2} x - 3600 \, \alpha x^{2} - 1200 \, x^{3}\Big).
\end{dmath*}

\bigskip
\bigskip

\section*{Acknowledgements}
The research of M. M. Castro was partially supported by  PID2021-124332NB-C21 (FEDER(EU) / Ministerio de Ciencia e Innovaci\'on-Agencia Estatal de Investigaci\'on) and FQM-262 (Junta de Andalucía).
The research of I. Zurri\'an was supported by CONICET PIP grant 112-200801-01533, SECyT  grant  30720150100255CB, Ministerio de Ciencia e Innovaci\'on PID2021-124332NB-C21 and Universidad de Sevilla VI PPIT - US.


\begin{thebibliography}{99}



\bibitem{CY1} Casper, W. R., Yakimov, M., {\em{ Integral operators, bispectrality and growth of Fourier algebras}}, J. Reine Angew. Math. {\bf 766} (2020), 151--194.

\bibitem{CY2} Casper, W. R., Yakimov, M., {\em{The matrix {B}ochner problem}}, Amer. J. Math. {\bf 144} (2022), no. 4, 1009--1065.






\bibitem{CGYZ1}
Casper, W. R., Gr{\"{u}}nbaum F. A., Yakimov,  M.,  Zurrian, I., {\em Reflective prolate-spheroidal operators and the KP/KdV equations}. Proc. Nat. Academy of Sciences, USA {\bf  116} (2019) , no. 37,  1831--18315.

\bibitem{CGYZ2}
Casper, W., Gr{\"{u}}nbaum, F. A., Yakimov,  M., Zurrian, I.,  {\em Reflective prolate-spheroidal operators and the adelic Grassmannian}. Comm. Pure and Applied Math (2023),1–42. DOI: 10.1002/cpa.22118



\bibitem{CGYZ3}
Casper, W., Gr{\"{u}}nbaum, F. A., Yakimov,  M.,  Zurrian, I.,
\newblock {\em Algebras of commuting differential operators for kernels of Airy type}.  Toeplitz operators and Random Matrices, in memory of Harold Widom. Birkhouser series: Operator Theory, Advances and Applications. E. Basor, A. B\"ottcher, T. Ehrhardt, C. Tracy editors, 2023.
 
\bibitem{CGYZ4}
Casper, W., Gr{\"{u}}nbaum, F. A., Yakimov,  M.,  Zurrian, I., Matrix valued discrete-continuous functions with the prolate spheroidal property and bispectrality (2023), arXiv:2302.05750.

\bibitem{CG5} Castro, M.,  Gr\"unbaum, F. A.,  {\it The Darboux process and time-and-band limiting for matrix orthogonal polynomials}.  Linear Algebra Appl.  {\bf 487} (2015),  328-341.

\bibitem{CG6} Castro, M., Gr\"unbaum, F. A.,  {\it Time and band limiting for matrix orthogonal polynomials of Jacobi type }. Random Matrices: Theory and Appl. {\bf 6}, no. 4 (2017): 1740001–12. 

	\bibitem{CGPZ}  Castro, M. ,Gr\"unbaum,   F.A.,Pacharoni, I., Zurri\'an, I., {\em A further look at time-and-band limiting for matrix orthogonal polynomials},  ``{\it Frontiers in Orthogonal Polynomials and q-Series}'', 139–153.
Contemp. Math. Appl. Monogr. Expo. Lect. Notes, 1, World Scientific 2018. 

\bibitem{CG23}
Castro, M. M., Gr\"unbaum, F. A.,
\newblock A new property of exceptional orthogonal polynomials (2023), arXiv:2210.13928.
 

\bibitem{CM}  Connes, A., Moscovici, H., {\em The UV prolate spectrum matches the zeros of zeta}. PNAS (2022) {\bf 119} (22). 


\bibitem{DG} Duistermaat J. J., Gr\"unbaum F. A., {\em
	Differential equations in the spectral parameter}. Comm. Math.
Phys. {\bf 103} (1986), 177--240.

\bibitem{Durec}  Dur\'an, A., {\em Higher order recurrence relation for Exceptional Charlier, Meixner, Hermite and Laguerre orthogonal polynomials}. Integral Transforms Spec. Funct. 26 (2015), no. 5, 357–376.

\bibitem{Du}  Dur\'an, A.  {\em  Exceptional orthogonal polynomials via Krall discrete polynomials}. Lectures on orthogonal polynomials and special functions, 1-75, London Math. Soc. Lecture Note Ser., 464, Cambridge U. Press, Cambridge 2021.

\bibitem{GGU2}  Garcia Ferrero, M.A., Gomez Ullate, D., Milson, R.  {\em A Bochner type classification theorem for exceptional orthogonal polynomials}. J. Math. Anal. Appl. 472 (2019) 584-626.
 

\bibitem{GUHerm} Gomez Ullate, D., Grandati, Y., N., Milson, R., {\em Rational extensions of the quantum harmonic oscillator and exceptional Hermite polynomials}. J. Phys. A 47 (2014), no. 1, 015203, 27 pp.

\bibitem{GU2}  Gomez Ullate, D., Kamran, N., Milson, R., {\em An extended class of orthogonal polynomials defined by  a Sturm-Liouville problem}. J. of Math. Analysis and Appl. {\bf 359} (2009) 352-367.

\bibitem{GUKKM}  Gomez-Ullate, D., Kasman, A., Kuijlaars, A.B.J.,  Milson, R., {\em Recurrence relations for exceptional Hermite polynomials}. J. of Approx. Theory {\bf 204} 2016,  1-16.

 

\bibitem{G3} Gr\"unbaum, F. A., {\em A new property of reproducing
	kernels of	classical orthogonal polynomials}. J. Math. Anal. Applic. {\bf 95} (1983),
491--500.
 


\bibitem{GPZ1}  Gr\"unbaum, F. A.,   Pacharoni I.,  Zurrian, I., {\it Time and band limiting for matrix valued functions, an example}. SIGMA  {\bf11} (2015), 044, 14 pages.

\bibitem{GPZ2}  Gr\"unbaum, F. A.,   Pacharoni I.,  Zurrian, I., {\it Time and band limiting for matrix valued functions: an integral and a commuting differential operator}. Inverse Problems {\bf 33}, No. 2 (2017), 025005.

\bibitem{GPZ3}  Gr\"unbaum, F. A.,   Pacharoni I.,  Zurrian, I., {\it Bispectrality and time and band limiting:  matrix valued polynomials}.  Int. Math. Res. Not. (2020), no. 13, 4016–4036.


\bibitem{GVZ}  Gr\"unbaum, F.A. , Vinet, L.,  Zhedanov, A., {\em Algebraic Heun operator and Band and Time limiting}. Com. Math. Physics {\bf 364} (2018), 1041--1068.

 



 

\bibitem{KM} Kasman, A.,  Milson R.,  {\em The adelic Grassmannian and exceptional Hermite polynomials},  Math Physics Anal. Geometry  {\bf23} 40 (2020), 51 pp. 


\bibitem{KV}  Katsnelson V., {\em Selfadjoint boundary conditions for the prolate spheroidal differential operator}, arXiv:1603.07542, (2016).




 



\bibitem{SLP2}  Landau, H. J.,  Pollak, H. O., {\em Prolate spheroidal	wave
	functions, Fourier Analysis and Uncertainty, II}. Bell System Tech. Journal,
Vol.~{\bf 40}, No.~1 (1961), 65--84.

\bibitem{SLP3} Landau, H. J.,  Pollak, H. O., {\em Prolate spheroical
	wave
	functions, Fourier Analysis and Uncertainty, III}. Bell System Tech. Journal,
Vol.~{\bf 41}, No.~4 (1962), 1295--1336.


\bibitem{ORokX} Osipov, A. Rokhlin, V.  Xiao H., \newblock {\em Prolate spheroidal wave functions of order zero. } \newblock {Springer Ser. Appl. Math. Sci, 187, 2013.}









 

\bibitem{Per}  Perline, R.K., {\it Discrete Time-Band Limiting Operators and Commuting Tridiagonal Matrices}. SIAM. J. on Algebraic and Discrete Methods, {\bf8}(2) (1987), 192--195.


\bibitem{P1} Perlstadt, M., {\em Chopped orthogonal polynomial
	expansions -- some
	discrete cases}. SIAM J. Disc. Meth. {\bf 4} (1983), 94--100.

\bibitem{P2} Perlstadt, M., {\em A property of orthogonal polynomial
	families
	with polynomial duals}. SIAM J. Math. Anal. {\bf 15} (1984), 1043--1054.

  

\bibitem{Qu}  Quesne, C., {\em  Solvable rational potentials and exceptional orthogonal polynomials in supersymmetric quantum mechanics}. SIGMA {\bf 5} (2009), 24 pp.


\bibitem{ORokX} Osipov, A. Rokhlin, V.  Xiao H., \newblock {\em Prolate spheroidal wave functions of order zero. } \newblock {Springer Ser. Appl. Math. Sci, 187, 2013.}

 
\bibitem{STZ}  Sasaki,R., Tsujimoto, S., Zhedanov, A., {\em Exceptional Laguerre and Jacobi polynomials and the corresponding potentials through Darboux-Crum transformations}. J. Phys. A. Math. {\bf 43} (2010), no. 31, 315204, 20 pp.


\bibitem{Sh} Shannon, C., {\em A mathematical theory of communication}. Bell Tech. J., vol {\bf 27}, 1948,  379--423 (July) and  623--656 (Oct).
 


\bibitem{SLP4} Slepian, D., {\em Prolate spheroidal wave functions, Fourier Analysis and Uncertainty, IV}. Bell System Tech. Journal, Vol.~{\bf 43}, No.~6
(1964), 3009--3058.

\bibitem{S2}  Slepian, D.,  {\em On bandwidth}. Proc. of IEEE, Vol. {\bf 64}, no. 3, March 1976.

\bibitem{SLP5}  Slepian, D., {\em Prolate spheroidal wave functions,
	Fourier Analysis and Uncertainty, V}. Bell System Tech. Journal, Vol.~{\bf 57}, No.~5 (1978),
1371--1430.

\bibitem{S1}  Slepian, D., {\em Some comments on Fourier analysis, Uncertainty and Modeling} SIAM Review, Vol. {\bf 25}, 3, July 1983, 379--393.

\bibitem{SLP1}  Slepian, D.,  Pollak H. O., {\em Prolate spheroidal wave
	functions, Fourier Analysis and Uncertainty, I}. Bell System Tech. Journal,
Vol.{\bf ~40},
No.~1 (1961), 43--64.
 

 

\end{thebibliography}
\end{document}